\documentclass[11pt,reqno]{amsart} 
\usepackage{geometry}
\usepackage{hyperref} 
\hypersetup{colorlinks=true, linkcolor=blue, citecolor=blue, filecolor=blue, urlcolor=blue}
\usepackage[comma, sort&compress, numbers]{natbib}

\usepackage{graphicx}
\newcommand{\R}{\mathbb{R}}

\newcommand{\N}{\mathbb{N}}

\newcommand{\EE}{\mathbb{E}}

\newcommand{\cL}{\mathcal{L}}
\renewcommand{\P}{\mathbb{P}}
\newcommand{\cD}{\mathcal{D}}
\newcommand{\cT}{\mathcal{T}}
\newcommand{\cR}{\mathcal{R}}

\newcommand{\eee}{\mathrm{e}}

\newcommand{\Wt}{\widetilde{\mathcal{W}}}

\newcommand{\1}{\mathbf{1}}

\newcommand{\h}{\widetilde{h}}
\newcommand{\W}{\mathcal{W}}

\newcommand{\ft}{\widetilde{f}}

\newcommand{\dx}{\,\mathrm{d}x}
\newcommand{\dy}{\,\mathrm{d}y}

\def\d{\mathrm{d}}
\def\e{\mathrm{e}} 
\def\N{\mathbb{N}}
\def\R{\mathbb{R}}

\def\dfrac#1#2{\lower0.15ex\hbox{\large$\frac{#1}{#2}$}}

\numberwithin{equation}{section}

 \makeatletter
 \@addtoreset{equation}{section}
 \makeatother

 \makeatletter
 \@addtoreset{enunciato}{section}
 \makeatother

 \newcounter{enunciato}[section]

 \newtheorem{ittheorem}{Theorem}
 \newtheorem{itlemma}{Lemma}
 \newtheorem{itproposition}{Proposition}
 \newtheorem{itdefinition}{Definition}
 \newtheorem{itremark}{Remark}
 \newtheorem{itclaim}{Claim}
 \newtheorem{itfact}{Fact}
 \newtheorem{itexample}{Example}
 \newtheorem{itconjecture}{Conjecture}
 \newtheorem{itobservation}{Observation} 
 \newtheorem{itcorollary}{Corollary} 
 \newtheorem{itquestion}{Question} 
 \newtheorem{itworkinghypothesis}{Working hypothesis}

 \newenvironment{theorem}{\addtocounter{enunciato}{1}
 \begin{ittheorem}}{\end{ittheorem}}

 \newenvironment{lemma}{\addtocounter{enunciato}{1}
 \begin{itlemma}}{\end{itlemma}}

 \newenvironment{proposition}{\addtocounter{enunciato}{1}
 \begin{itproposition}}{\end{itproposition}}

 \newenvironment{definition}{\addtocounter{enunciato}{1}
 \begin{itdefinition}}{\end{itdefinition}}

 \newenvironment{remark}{\addtocounter{enunciato}{1}
 \begin{itremark}}{\end{itremark}}

 \newenvironment{claim}{\addtocounter{enunciato}{1}
 \begin{itclaim}}{\end{itclaim}}

 \newenvironment{fact}{\addtocounter{enunciato}{1}
 \begin{itfact}}{\end{itfact}}

 \newenvironment{conjecture}{\addtocounter{enunciato}{1}
 \begin{itconjecture}}{\end{itconjecture}}

 \newenvironment{corollary}{\addtocounter{enunciato}{1}
 \begin{itcorollary}}{\end{itcorollary}}

 \newcommand{\be}[1]{\begin{equation}\label{#1}}
 \newcommand{\ee}{\end{equation}}

 \newcommand{\bl}[1]{\begin{lemma}\label{#1}}
 \newcommand{\el}{\end{lemma}}

 \newcommand{\br}[1]{\begin{remark}\label{#1}}
 \newcommand{\er}{\end{remark}}

 \newcommand{\bt}[1]{\begin{theorem}\label{#1}}
 \newcommand{\et}{\end{theorem}}

 \newcommand{\bd}[1]{\begin{definition}\label{#1}}
 \newcommand{\ed}{\end{definition}}

 \newcommand{\bcl}[1]{\begin{claim}\label{#1}}
 \newcommand{\ecl}{\end{claim}}

 \newcommand{\bfact}[1]{\begin{fact}\label{#1}}
 \newcommand{\efact}{\end{fact}}

 \newcommand{\bp}[1]{\begin{proposition}\label{#1}}
 \newcommand{\ep}{\end{proposition}}

 \newcommand{\bc}[1]{\begin{corollary}\label{#1}}
 \newcommand{\ec}{\end{corollary}}

 \newcommand{\bcj}[1]{\begin{conjecture}\label{#1}}
 \newcommand{\ecj}{\end{conjecture}}

 \newcommand{\bpr}{\begin{proof}}
 \newcommand{\epr}{\end{proof}}

 \newcommand{\bprl}[1]{\begin{proofof}{\it\ref{#1}}.\,\,}
 \newcommand{\eprl}{\end{proofof}}

 \newcommand{\bi}{\begin{itemize}}
 \newcommand{\ei}{\end{itemize}}

 \newcommand{\ben}{\begin{enumerate}}
 \newcommand{\een}{\end{enumerate}}

 \allowdisplaybreaks

\begin{document}

\title[LDP Laplacian]{Large deviation principle for the norm of the Laplacian 
matrix of inhomogeneous Erd\H{o}s-R\'enyi random graphs\thanks{The research in this paper was supported through NWO Gravitation Grant NETWORKS 024.002.003.}}
     
\author[R.~S.~Hazra]{Rajat Subhra Hazra}
\author[F.\ den Hollander]{Frank den Hollander}
\author[M.\ Markering]{Maarten Markering}
\address{Mathematisch instituut, Universiteit Leiden, The Netherlands}
\email{r.s.hazra@math.leidenuniv.nl}
\email{denholla@math.leidenuniv.nl}
\address{DPMMS, University of Cambridge, United Kingdom}
\email{mjrm2@cam.ac.uk}

\begin{abstract}
We consider an inhomogeneous Erd\H{o}s-R\'enyi random graph $G_N$ with vertex set $[N] = \{1,\dots,N\}$ for which the pair of vertices $i,j \in [N]$, $i\neq j$, is connected by an edge with probability $r_N(\tfrac{i}{N},\tfrac{j}{N})$, independently of other pairs of vertices. Here, $r_N\colon\,[0,1]^2 \to (0,1)$ is a symmetric function that plays the role of a reference graphon. Let $\lambda_N$ be the maximal eigenvalue of the Laplacian matrix of $G_N$. We show that if $\lim_{N\to\infty} \|r_N-r\|_\infty = 0$ for some limiting graphon $r\colon\,[0,1]^2 \to (0,1)$, then $\lambda_N/N$ satisfies a downward LDP with rate $\binom{N}{2}$ and an upward LDP with rate $N$. We identify the associated rate functions $\psi_r$ and $\widehat{\psi}_r$, and derive their basic properties.
\end{abstract}
\maketitle
\newcommand{\ABS}[1]{\left(#1\right)} 
\newcommand{\veps}{\varepsilon} 



\section{Introduction and main results}

Section~\ref{subsec:background} provides background. Section~\ref{subsec:setting} states the LDP for the empirical graphon associated with inhomogeneous Erd\H{o}s-R\'enyi random graphs. Section~\ref{subsec:grop} looks at graphon operators, in particular, the Laplacian operator that is the central object in the present paper. Sections~\ref{subsec:LDPdown}--\ref{subsec:LDPup} state the downward, respectively, upward LDP for the largest eigenvalue of the Laplacian matrix and present basic properties of the associated rate functions. Section~\ref{subsec:disc} places the various theorems in their proper context.    


\subsection{Background}
\label{subsec:background}

Spectra of matrices associated with a graph play a crucial role in understanding the geometry of the graph. Given a finite graph on $N$ vertices, two important matrices are the \emph{adjacency matrix} $A_N$ and the \emph{Laplacian matrix} $L_N= D_N- A_N$, where $D_N$ is the diagonal matrix whose elements are the degrees of the vertices. In this paper we focus on the largest eigenvalue of $L_N$ when the underlying graph is a \emph{dense inhomogeneous} Erd\H{o}s-R\'enyi random graph. The largest eigenvalue of $A_N$ satisfies a large deviation principle (LDP). This fact is an immediate consequence of the LDP for the empirical graphon derived in \cite{DS19} and \cite{M20} in combination with the contraction principle, because the norm of the adjacency graphon operator is bounded and continuous on the space of graphons. The rate function is given in terms of a variational formula involving the rate function of the LDP for the empirical graphon. In \cite{CHHS20} we analysed this variational formula, identified the basic properties of the rate function, and identified its scaling behaviour near its unique minimiser and its two boundary points. 

The extension of the LDP to $L_N$ poses new challenges, because $L_N$ is a more delicate object than $A_N$. For one, the upward and the downward large deviations for the largest eigenvalue of $L_N$ live on \emph{different scales}, and the norm of the Laplacian graphon operator lacks certain continuities properties that hold for the norm of the adjacency graphon operator (see Remark \ref{rem:continuity}). Like for $A_N$, it is not possible to explicitly solve the variational formulas for the two associated rate functions. Nonetheless, we derive their basic properties and identify their scaling behaviour near their unique minimisers.
Our proofs require an analysis of the functional analytic properties of the Laplacian operator, in combination with variational tools. 

The literature on the eigenvalues of the Laplacian of a random matrix is limited. For a general Wigner matrix with independent centered entries satisfying certain moment conditions, the empirical spectral distribution was derived in \cite{bryc} and was identified as the free additive convolution of the semicircle law and the Gaussian distribution. For entries with a common mean $m$ it was shown that the empirical spectral distribution converges to the Dirac measure at $m$. Moreover, under a fourth moment condition it was shown that the spectral norm is of order $\sqrt{2N\log N}$ when the entries are centered, and of order $N$ when the entries are not centered. Although the moment conditions were rather general, \cite{bryc} did not cover all relevant cases of Erd\H{o}s-R\'enyi random graphs. This was resolved in \cite{Jiang2012}, \cite{DingJiang} for the case where the average degree diverges with $N$. The results were extended to inhomogeneous Erd\H{o}s-R\'enyi graphs in \cite{chakrabarty2019}, \cite{chatterjee-hazra}. The empirical distribution in the sparse homogeneous Erd\H{o}s-R\'enyi graph was identified in \cite{Khorunzhyetal}. Later, an alternative proof through local weak convergence was provided in \cite{BordenaveLelarge}. The results on extreme eigenvalues are limited and challenging. Recently, it was shown in \cite{campbell2022extreme} that for the Gaussian orthogonal ensemble the largest eigenvalue of the Laplacian, after appropriate scaling and centering, converges to the Gumbel distribution. 

Large deviations for Erd\H{o}s-R\'enyi random graphs were explored extensively in various works, including \cite{chatterjeeln2017}, \cite{CV11}, \cite{LZ16} by utilising graphon theory, specifically subgraph densities and maximal eigenvalues. For a comprehensive review of the literature, we refer to \cite{chatterjeeln2017}. Large deviation theory for random matrices originated in \cite{arous:guionnet}, where the focus was on the empirical spectral distribution of $\beta$-ensembles with a quadratic potential. The rate was found to be $N^2$, and the rate function was identified by using a non-commutative notion of entropy. Subsequently, the maximal eigenvalue of such ensembles was investigated in \cite{arous:guionnet:dembo}. In \cite{bordenavecaputo2014}, large deviations of the empirical spectral distribution were derived for random matrices with non-Gaussian tails, while the study of the maximal eigenvalue in that context was carried out in \cite{augeri2016}, \cite{augeriguionnethusson2019}. However, the adjacency matrix and the Laplacian matrix of an inhomogeneous Erd\H{o}s-R\'enyi random graph fall outside these regimes. The goal of the present paper is to understand the \emph{large deviations of the largest eigenvalue} of an \emph{inhomogeneous Erd\H{o}s-R\'enyi graph} under certain density assumptions.


\subsection{LDP for inhomogeneous Erd\H{o}s-R\'enyi random graphs}
\label{subsec:setting}


\subsubsection{Graphons} 

Let
\begin{equation}
\mathcal{W} = \bigl\{h\colon\, [0,1]^2 \to [0,1]\colon\,h(x,y) = h(y,x)\,\, \forall \, x,y \in [0,1] \bigr\}
\end{equation}
denote the set of graphons. Let $\mathcal{M}$ denote the set of Lebesgue measure-preserving bijective maps $\phi\colon\, [0,1] \mapsto [0,1]$. For two graphons $h_1, h_2 \in \mathcal{W}$, the \emph{cut-distance} is defined by 
\begin{equation}
\label{cutdistance}
d_{\square}(h_1, h_2) = \sup_{S,T \subset [0,1]} \bigg| \int_{S \times T}\d x \, \d y \, \big[h_1(x,y) - h_2(x,y) \big]  \bigg|,
\end{equation}
and the \emph{cut-metric} by
\begin{equation}
\delta_{\square}(h_1, h_2) = \inf_{\phi \in \mathcal{M}} d_{\square}(h_1, h_2^{\phi}),
\end{equation}
where $h_2^{\phi}(x,y) = h_2(\phi(x), \phi(y))$. The cut-metric defines an equivalence relation $\sim$ on $\mathcal{W}$ by declaring $h_1 \sim h_2$ if and only if $\delta_{\square}(h_1, h_2)= 0$, and leads to the quotient space $\widetilde{\mathcal{W}} = \mathcal{W}/_\sim$. For $h \in \mathcal{W}$, we write $\widetilde{h}$ to denote the equivalence class of $h$ in $\widetilde{\mathcal{W}}$. The equivalence classes correspond to relabelings of the vertices of the graph. The pair $(\widetilde{\mathcal{W}}, \delta_{\square})$ is a compact metric space \cite{L12}. 


\subsubsection{LDP for empirical graphon} 

Consider a graphon $r\in\mathcal{W}$ that plays the role of a \emph{reference graphon} and satisfies
\begin{equation}
\label{Assbasic} 
\log r, \, \log(1-r) \in L^1([0,1]^2).
\end{equation}
Consider further a sequence of graphons $r_N\colon[0,1]^2\to(0,1)$, $N \in \N$, that are constant on the blocks $[\frac{i-1}{N},\frac{i}{N})\times[\frac{j-1}{N},\frac{j}{N})$, $1\leq i,j\leq N$, and satisfy
\begin{equation}
\label{Assbasic2}
\begin{split}
&\lim_{N\to\infty} r_N=r\quad\text{a.e.},\\
&\lim_{N\to\infty} \|r_N - r\|_1 = 0,\\
&\lim_{N\to\infty} \|\log r _ N-\log r\|_1 = 0,\\
&\lim_{N\to\infty} \|\log(1-r_N) - \log(1-r)\|_1 = 0.
\end{split}
\end{equation}
Let $G_N$ be the inhomogeneous Erd\H{o}s-R\'enyi random graph with vertex set given by $[N] = \{1,\dots,N\}$ for which the pair of vertices $i,j \in [N]$, $i\neq j$, is connected by an edge with probability $(r_N)_{ij}$, independently of other pairs of vertices, where $(r_N)_{ij}$ is the value of $r_N$ on the block $[\frac{i-1}{N},\frac{i}{N})\times[\frac{j-1}{N},\frac{j}{N})$. The function $r_N$ plays the role of a \emph{block reference graphon} converging to some \emph{reference graphon} $r$ as $N\to\infty$.

Let $A_N$ be the \emph{adjacency matrix} of $G_N$ defined by
\begin{equation}
A_N(i,j)
= \left\{\begin{array}{ll}
1, &\text{if there is an edge between vertex $i$ and vertex $j$},\\
0, &\text{otherwise}. 
\end{array}
\right.
\end{equation} 
Let $D_N$ be the diagonal degree matrix defined by $D_N(i,i) = \sum_{j \in [N] \setminus \{i\}} A_N(i,j)$ and $D_N(i,j) = 0$ for $i \neq j$, and put $L_N = D_N - A_N$, which is the \emph{Laplacian matrix}. Write $\mathbb{P}_N$ to denote the law of $G_N$. Use the same symbol for the law on $\mathcal{W}$ induced by the map that associates with the graph $G_N$ its empirical graphon $h^{G_N}$, defined by
\begin{equation}
h^{G_N}(x,y)
= \left\{\begin{array}{ll}
1, &\text{if there is an edge between vertex $\lceil Nx\rceil$ and vertex $\lceil Ny\rceil$},\\
0, &\text{otherwise}. 
\end{array}
\right. 
\end{equation}
Write $\widetilde{\mathbb{P}}_N$ to denote the law of $\widetilde{h}^{G_N}$.

The following LDP, which is an extension of the celebrated LDP for homogeneous ERRG derived in \cite{CV11}, is proved in \cite{DS19} and \cite{M20}. (For more background on large deviation theory, see, for instance, \cite{H00}.) 

\begin{theorem}{\bf [LDP for inhomogeneous ERRG, \cite{DS19}, \cite{M20}]}
\label{thm:LDPinhom}
Subject to \eqref{Assbasic}--\eqref{Assbasic2}, the sequence $(\widetilde{\mathbb{P}}_N)_{N\in\mathbb{N}}$ satisfies the large deviation principle on $(\widetilde{\mathcal{W}},\delta_{\square})$ with rate $\binom{N}{2}$, i.e., 
\begin{equation}
\begin{aligned}
\limsup_{N \to \infty} \binom{N}{2}^{-1} \log \widetilde{\mathbb{P}}_N (\mathcal{C}) 
&\leq - \inf_{\widetilde h \in \mathcal{C}} J_r(\widetilde h)
&\forall\,\mathcal{C} \subset \widetilde{\mathcal{W}} \text{ closed},\\ 
\liminf_{N \to \infty} \binom{N}{2}^{-1} \log \widetilde{\mathbb{P}}_N(\mathcal{O}) 
&\geq - \inf_{\widetilde h \in \mathcal{O}} J_r(\widetilde h)
&\forall\,\mathcal{O} \subset \widetilde{\mathcal{W}} \text{ open},
\end{aligned}
\end{equation}
where the rate function $J_r\colon\,\widetilde{\mathcal{W}} \to \mathbb{R}$ is given by
\begin{equation}
\label{Jrhdef}
J_r(\widetilde h) = \inf_{\phi \in \mathcal{M}} I_r(h^\phi),
\end{equation}
where $h$ is any representative of $\widetilde h$ and 
\begin{equation}
\label{Irhdef}
I_r(h) = \int_{[0,1]^2} \d x\, \d y \,\, \mathcal{R}\big(h(x,y) \mid r(x,y)\big), \quad h \in \mathcal{W},
\end{equation}
with
\begin{equation}
\label{Rdef}
\mathcal{R}\big(a \mid b\big) = a \log \tfrac{a}{b} + (1-a) \log \tfrac{1-a}{1-b}
\end{equation}
the relative entropy of two Bernoulli distributions with success probabilities $a \in [0,1]$, $b \in (0,1)$ (with the convention $0 \log 0 = 0$).
\end{theorem}

\begin{remark}
{\rm Theorem~\ref{thm:LDPinhom} was proved in \cite{DS19} under the assumption that $r$ is bounded away from $0$ and $1$. In \cite{M20} this assumption was relaxed to \eqref{Assbasic}, and it was also shown that $J_r$ is a good rate function, i.e., $J_r \not\equiv \infty$ and $J_r$ has compact level sets. Note that \eqref{Jrhdef} differs from the expression in \cite{DS19}, where the rate function is the lower semi-continuous envelope of $I_r(h)$. However, as shown in \cite{M20}, under \eqref{Assbasic} the two rate functions are equivalent, since $J_r(\widetilde h)$ is lower semi-continuous on $\widetilde{\mathcal{W}}$.}\hfill$\spadesuit$
\end{remark} 

As an application of Theorem~\ref{thm:LDPinhom}, it was shown in \cite{CHHS20} that the largest eigenvalue of the adjacency matrix satisfies the LDP with rate ${N\choose 2}$. The rate function was analysed in detail for reference graphons that are rank-1. In the present paper we focus on the LDP for the largest eigenvalue of the Laplacian matrix $L_N$. 


\subsection{Graphon operators}
\label{subsec:grop}

For $h \in \mathcal{W}$, the graphon operator $\cT_h$ is the integral operator on $L^2([0,1])$ defined by 
\begin{equation}
(\cT_hu)(x)=\int_{[0,1]}\d y\,h(x,y)u(y), \qquad x \in [0,1].
\end{equation}
Note that $\cT_h$ is a compact operator. Define the \emph{degree function} as 
\begin{equation}
\label{degreefunction}
d_h(x)= \int_{[0,1]} \d y\,h(x,y), \qquad x \in [0,1].
\end{equation} 
The \emph{degree operator} $\cD_h$ is the multiplication operator on $L^2([0,1])$ defined by
\begin{equation}
\label{degreeoperator}
 (\cD_hu)(x) = d_h(x) u(x),
\end{equation}
The \emph{Laplacian operator} $\cL_h$ is the linear integral operator on $L^2([0,1])$ defined by
\begin{equation} 
\label{graphonoperator}
(\cL_h u)(x) = \int_{[0,1]} \d y \,h(x,y) [u(x)-u(y)], \qquad x \in [0,1].
\end{equation}
Note that 
\begin{equation}
\label{Laplaciandef}
\cL_h = \cD_h-\cT_h.
\end{equation}  

Recall that, given an operator $S$ on a Hilbert space, the \emph{spectrum} of $S$ is defined as
\begin{equation}
\sigma(S)= \{\lambda\in \mathbb C\colon\, S - \lambda I \text{ is not invertible}\}.
\end{equation} 
Let $\sigma_d(S)$ denote the \emph{discrete spectrum} of $S$, which consists of all the isolated eigenvalues with finite algebraic multiplicity. The \emph{essential spectrum} of $S$ is denoted by
\begin{equation}
\sigma_{\mathrm{ess}}(S) = \sigma(S)\setminus \sigma_d(S).
\end{equation}
The essential spectrum is closed, and the discrete spectrum can only have accumulation points on the boundary of the essential spectrum. Since it is known that compact operators do not affect the essential spectrum, we have
\begin{equation}
\sigma_{\mathrm{ess}}(\cL_h) = \sigma_{\mathrm{ess}}(\cD_h).
\end{equation}
See \cite[Theorem IV.5.35]{K66} for more details.

The operator $\cL_h$ is not as well-behaved as $\cT_h$ with the cut-norm. In fact, even when a sequence of graphons $(h_n)_{n\in\N}$ converges in cut-norm to a graphon $h$, the eigenvalues and eigenvectors of $\cL_{h_n}$ may not converge to those of $\cL_h$, as was already observed in \cite{DGKR16}, \cite{LBB2008}. Assuming as in \cite{CHHS20} that the reference graphon is rank-1 does not help. In fact, if we assume that $r$ is rank-1 and is continuous, then $0$ is the only eigenvalue of $\cL_r$ and $\sigma_{\mathrm{ess}}(\cL_r)= d_r([0, 1])$, as shown in \cite[Proposition 5.11]{DGKR16}.

If $h$ is the empirical graphon of a graph $G$ with $N$ vertices, then $N\|\cT_h\|$ equals the largest eigenvalue of the adjacency matrix of $G$, and $N\|\cD_h\|$ equals the maximum degree of $G$. In fact, for any graphon $h$ the spectrum of $\cD_h$ equals the range of $d_h$ and the operator norm of $\cD_h$ equals the supremum norm $d_h$, i.e.,
\begin{equation}
\label{eq:degreenorm}
\|\cD_h\|=\| d_h\|_{\infty},
\end{equation}
where $\|d_h\|_{\infty}$ is the $L^\infty$-norm of the function $d_h$. This follows from the fact that $\cD_h$ is a multiplication operator. Let 
\begin{equation}
\label{normgraphon}
\|\cL_h\| = \sup_{ \substack{u \in L^2([0,1])\\ \| u \|_2 = 1} } \| \cL_hu \|_2
\end{equation}
be the operator norm of $\cL_h$, where $\| \cdot \|_2$ denotes the $L^2$-norm. Since $\cL_h$ is a normal operator with a non-negative spectrum, $\|\cL_h\|$ also equals the supremum of the spectrum of $\cL_h$. 

\begin{proposition}{\bf [Properties of the Laplacian operator]}
\label{prop:basics}
$\mbox{}$\\
(i) Let $h$ be a graphon and $\cL_h$ be the Laplacian operator on $L^2([0,1])$. Then $\cL_h$ is a bounded operator, and $h \mapsto \|\cL_h\|$ is lower semi-continuous in the cut-metric.\\
(ii) Let $G$ be a graph with $N$ vertices and $h^{G}$ be the empirical graphon associated with $G$ and let $L_N$ be the Laplacian matrix with spectral norm $\|L_N\|$. Then it follows that 
\begin{equation}
\frac{\|L_N\|}{N} = \|\cL_{h^{G}}\| \qquad \forall\,N.
\end{equation}
\end{proposition}

\begin{remark}\label{rem:continuity}
{\rm Note that $h \mapsto \|\cL_h\|$ is not continuous in the cut-metric. For example, consider the sequence of graphons $(h_N)_{N\in\N}$ such that $h_N$ is the empirical graphon of the $N$-star graph (i.e., $1$ vertex connected by an edge to each of the $N-1$ other vertices, and no further edges). Then $h_N \downarrow 0$ as $N\to\infty$ in the cut-metric, but $\|\cL_{h_N}\|=1$ for all $N\in\N$.} \hfill$\spadesuit$
\end{remark}

Proposition~\ref{prop:basics} is proven in Section~\ref{sec:proofprop}.


\subsection{Main theorems: downward large deviations} 
\label{subsec:LDPdown}

Let 
\begin{equation}
\lambda_N = \|L_N\| = \sup_{ \substack{u \in L^2([0,1])\\ \|u\|_2 = 1}} \|L_N u\|_2.
\end{equation} 
be the maximal eigenvalue of $L_N$, where $\|\cdot\|_2$ denotes the $L^2$-norm. Abbreviate
\begin{equation}
C_r = \|\cL_r\|.
\end{equation} 
Our goal is to show that $\lambda_N/N$ satisfies a downward LDP as $N \to \infty$, with rate $\binom{N}{2}$ and with a rate function that can be analysed in detail. We write $\mathbb{P}^*_N$ to denote the law of $\lambda_N$.

\begin{theorem}{\bf [Downward LDP]}
\label{thm:main}
Subject to \eqref{Assbasic}--\eqref{Assbasic2},
\begin{equation}
\lim_{N\to\infty} \binom{N}{2}^{-1} \log  \mathbb{P}^*_N(\lambda_N/N \leq \beta)
= - \psi_r(\beta) , \qquad \beta \in [0,C_r],
\end{equation}
with
\begin{equation}
\label{rf}
\psi_r(\beta)  = \inf_{ \substack{\widetilde{h} \in \widetilde{\mathcal{W}}\\ \|\cL_{\widetilde{h}}\| \leq \beta} } J_r(\widetilde{h})
= \inf_{ \substack{h \in \mathcal{W}\\ \|\cL_h\| \leq \beta} } I_r(h).
\end{equation}
\end{theorem}

\noindent
The second equality in \eqref{rf} uses that $\|\cL_r\| = \|\cL_{r^\phi}\|$ for any $\phi \in \mathcal{M}$, as is evident after replacing $u$ by $u^{\phi^{-1}}$ in \eqref{normgraphon} given by $u^{\phi^{-1}}(x) = u(\phi^{-1}(x))$. Since the maximal eigenvalue is invariant under relabelling of the vertices, we need not worry about the equivalence classes. 

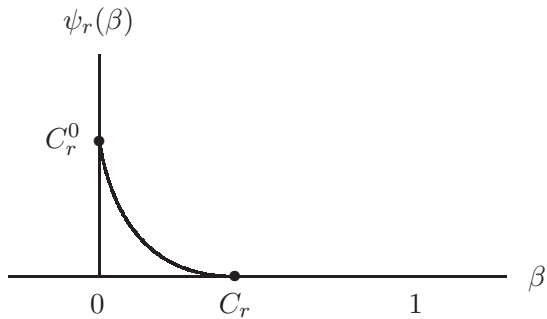
\begin{figure}[htbp]
\begin{center}
\setlength{\unitlength}{0.6cm}
\begin{picture}(10,6)(0,0)
{\thicklines
\qbezier(0,0)(0,2)(0,4.9)
\qbezier(-2,0)(4,0)(9,0)
\qbezier(3,0)(5,0)(7,0)
\qbezier(3,0)(0.5,0)(0,3)
}
\put(2.65,-.8){$C_r$}
\put(6.85,-.8){$1$}
\put(-0.2,-.8){$0$}
\put(-1.2,2.85){$C^0_r$}
\put(9.5,-.2){$\beta$}
\put(-.8,5.5){$\psi_r(\beta)$}
\put(0,3){\circle*{0.25}}
\put(3,0){\circle*{0.25}}
\end{picture}
\end{center}
\vspace{0.5cm}
\caption{\small Graph of $\beta \mapsto \psi_r(\beta)$.}
\label{fig-rfdown}
\vspace{0.2cm}
\end{figure}

Let 
\begin{equation}
C_r^0 = \int_{[0,1]^2} \log \tfrac{1}{1-r}. 
\end{equation}
When $\beta = C_r$, the optimal graphon is the reference graphon $r$ almost everywhere, for which $I_r(r) = 0$, and no large deviation occurs. When $\beta = 0$, the optimal graphon is the zero graphon $\underline{0} \equiv 0$, for which $I_r(\underline{0}) = C_r^0$ (see Fig.~\ref{fig-rfdown}).

\begin{theorem}{\bf [Properties of the rate function]}
\label{thm:rfprop}
Subject to \eqref{Assbasic}--\eqref{Assbasic2}:\\
(i) $\psi_r$ is continuous and strictly decreasing on $[0,C_r]$, with $\psi_r(0) = C^0_r > 0$ and $\psi_r(C_r) = 0$.\\
(ii) For every $\beta \in [0,C_r]$, the set of minimisers of the variational formula for $\psi_r(\beta)$ in \eqref{rf} is non-empty 
and compact in $\widetilde{\mathcal{W}}$.
\end{theorem}

Under the \emph{additional assumptions} that
\begin{eqnarray}
\label{gap}
&\lambda_{\max}(\cL_r)<\|d_r\|_{\infty},\\
\label{rbd1}
&\text{$r$ is bounded away from 0 and $1$}
\end{eqnarray}
we are able to compute the behaviour of $\psi_r$ around $C_r$. Here $\lambda_{\max}(\cL_r)$ denotes the largest eigenvalue of $\cL_r$. Note that \eqref{rbd1} is stronger than \eqref{Assbasic}.  Recall that $d_r(x) = \int_{[0,1]} \d y\,r(x,y)$, $x \in [0,1]$. It is easy to check that, subject to \eqref{gap},
\begin{equation}
C_r = \|d_r\|_\infty.
\end{equation}

\begin{theorem}{\bf [Scaling of the rate function]}
\label{thm:rfscaling}
Subject to \eqref{gap}--\eqref{rbd1},  
\begin{equation}
\label{scalpsir}
\psi_r(\beta) \asymp \int_{S_r(\beta)} \d x\,\frac{1}{v_r(x)}\,(d_r(x)-\beta)^2,
\qquad \beta \uparrow C_r,
\end{equation}
where 
\begin{equation}
S_r(\beta) = \{x \in [0,1]\colon\, d_r(x) \geq \beta\}
\end{equation}
and
\begin{equation}
\label{vrdef}
v_r(x) = \int_{[0,1]} \d y\,r(x,y)[1-r(x,y)]. 
\end{equation}
\end{theorem}

The proofs of Theorems~\ref{thm:main}--\ref{thm:rfscaling} are given in Sections~\ref{sec:proofmain}--\ref{sec:proofrfscaling}.


\subsection{Main theorems: upward large deviations} 
\label{subsec:LDPup}

Put
\begin{equation}
\label{Jrdef}
J_r(x,\beta) = \int_{[0,1]} \d y\, \mathcal{R}\big(\widehat{r}_\beta(x,y)\mid r(x,y)\big), \qquad x \in [0,1], 
\end{equation} 
where
\begin{equation}
\label{Cramer}
\widehat{r}_\beta(x,y)=\frac{\e^{\theta(x,\beta)} r(x,y)}{\e^{\theta(x,\beta)} r(x,y)+[1-r(x,y)]}
\end{equation}
is a \emph{Cram\'er-type transform} of the reference graphon $r$, with a Langrange multiplier function $\theta(x,\beta)$, $x \in [0,1]$, chosen such that $\int_{[0,1]} \d y\,\widehat{r}_\beta(x,y) = \beta$, $x \in [0,1]$. Note that $\widehat{r}_\beta$ does not need to be symmetric and therefore is not necessarily a graphon. Under the \emph{additional assumptions} that
\begin{eqnarray}
\label{everywhere}
&\lim_{N\to\infty}r_N(x,y) = r(x,y) \quad \forall\, (x,y)\in[0,1]^2,\\
\label{infconv}
&\lim_{N\to\infty}\| r_N-r\|_\infty = 0,\\
\label{posdef}
&\text{$r$ is non-negative definite (i.e., $\cT_r$ is a non-negative definite operator)}
\end{eqnarray}
we are able to derive an upward LDP and identify the associated rate function in term of $J_r$. Note that \eqref{rbd1} and \eqref{everywhere}--\eqref{infconv} together are stronger than \eqref{Assbasic}--\eqref{Assbasic2}.

\begin{theorem}{\bf [Upward LDP]}
\label{thm:mainalt}
Subject to \eqref{rbd1} and \eqref{everywhere}--\eqref{posdef},
\begin{equation}
\lim_{N\to\infty} N^{-1} \log  \mathbb{P}^*_N(\lambda_N/N \geq \beta)
= - \widehat{\psi}_r(\beta) , \qquad \beta \in [C_r,1],
\end{equation}
with
\begin{equation}
\label{rfalt}
\widehat{\psi}_r(\beta) = \inf_{x\in[0,1]} J_r(x,\beta).
\end{equation} 
\end{theorem}

\begin{figure}[htbp]
\begin{center}
\setlength{\unitlength}{0.6cm}
\begin{picture}(10,6)(0,0)
{\thicklines
\qbezier(0,0)(0,2)(0,4.9)
\qbezier(-2,0)(4,0)(9,0)
\qbezier(3,0)(5,0)(7,0)
\qbezier(3,0)(5.5,0)(7,3)
}
\qbezier[30](7,3)(7,1.5)(7,0)
\qbezier[40](0,3)(3,3)(7,3)
\put(2.65,-.8){$C_r$}
\put(6.85,-.8){$1$}
\put(-0.2,-.8){$0$}
\put(-1.2,2.85){$C^1_r$}
\put(9.5,-.2){$\beta$}
\put(-.8,5.5){$\widehat{\psi}_r(\beta)$}
\put(7,3){\circle*{0.25}}
\put(3,0){\circle*{0.25}}
\end{picture}
\end{center}
\vspace{0.5cm}
\caption{\small Graph of $\beta \mapsto \widehat{\psi}_r(\beta)$.}
\label{fig-rfup}
\vspace{0.2cm}
\end{figure}
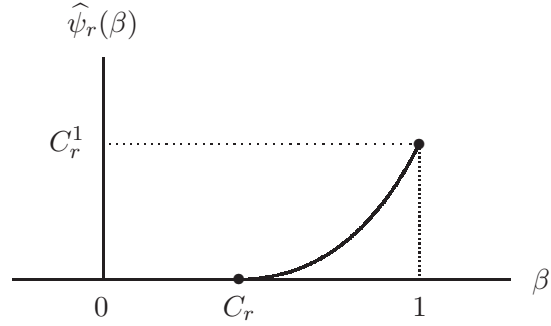

Define
\begin{equation}
C_r^1 = \widehat\psi_r(1) = \inf_{x\in[0,1]} \int_{[0,1]} \d y\,\log\frac{1}{r(x,y)}.
\end{equation}
When $\beta = C_r$, the Lagrange multiplier in \eqref{Cramer} is $\theta(x,C_r) \equiv 0 $, for which $J_r(x,C_r) \equiv 0$, and no large deviation occurs. When $\beta = 1$, the Lagrange multiplier is $\theta(x,1) \equiv \infty$, for which $\widehat{\psi}_r(1) = C_r^1$ (see Fig.~\ref{fig-rfup}).

\begin{theorem}{\bf [Properties of the rate function]}
\label{thm:rfpropalt}
Subject to \eqref{rbd1}, $\widehat\psi_r$ is continuous and strictly increasing on $[C_r,1]$, with $\widehat\psi_r(C_r)=0$ and $\widehat\psi_r(1) = C_r^1 > 0$.
\end{theorem}

In order to study the scaling of $\widehat{\psi}_r$ near $C_r$, we need assumption that
\begin{equation}
\label{rct}
\text{$r$ is continuous.}
\end{equation}

\begin{theorem}{\bf [Scaling of the rate function]}
\label{thm:rfscalingalt}
Subject to \eqref{rbd1}, \eqref{posdef} and \eqref{rct},
\begin{equation}
\label{scal1alt}
\widehat{\psi}_r(\beta) \sim \widehat{K}_r(\beta-C_r)^2, \qquad \beta \downarrow C_r,  
\end{equation}
with 
\begin{equation}
\label{curvature}
\widehat{K}_r = \frac{1}{2} \inf_{x \in \cD_r} \frac{1}{v_r(x)},
\end{equation}
where $\mathcal{D}_r = \{x \in [0,1]\colon\,d_r(x) = \|d_r\|_\infty\}$.
\end{theorem}

The proofs of Theorems~\ref{thm:mainalt}--\ref{thm:rfscalingalt} are given in Sections~\ref{sec:proofmainalt}--\ref{sec:proofrfscalingalt}.


\subsection{Discussion}
\label{subsec:disc}

$\mbox{}$

\medskip\noindent
{\bf 1.} Theorems~\ref{thm:main}--\ref{thm:rfscaling} establish the downward LDP and identify basic properties of the rate function $\psi_r$. Note that $\lim_{\beta \uparrow C_r} S_r(\beta) = \cD_r$. Hence, if $|\cD_r| = 0$, then the scaling of $\psi_r$ near $C_r$ is \emph{faster than quadratic}, which suggest that $\widehat{\lambda}_N/N$ does \emph{not} satisfy a standard central limit theorem. On the other hand, if $|\cD_r| > 0$, then the scaling is quadratic and a standard central limit is expected to hold. Both questions are open. Several scenarios for the precise scaling are possible depending on how $d_r$ scales near its maxima.

\medskip\noindent
{\bf 2.}
Assumptions \eqref{Assbasic}--\eqref{Assbasic2} are basic because they underly Theorem~\ref{thm:LDPinhom}, which is the jump board for our downward LDP in Theorem~\ref{thm:main} and the general properties of the downward rate function in Theorem~\ref{thm:rfprop}. Assumption~\eqref{gap} is needed for the upper bound in the scaling of the downward rate function in Theorem~\ref{thm:rfscaling}. It is the most severe assumption in the present paper, although it is still satisfied for a large class of graphons, for instance, those $r$ that are rank-1 and continuous (see \cite[Proposition 5.11]{DGKR16}). Assumption~\eqref{gap} guarantees that, for all graphons $h$ close enough to $r$ in $L^2$-norm, $\|\cL_h\|=\|d_h\|_{\infty}$. This allows us to reduce the analysis of the rate function for the LDP of the Laplacian norm to the analysis of the rate function for the maximum degree, which is considerably easier. Assumption~\eqref{rbd1} implies that $v_r(x) \geq \delta d_r(x)$, $x \in [0,1]$ for some $\delta>0$ (via \eqref{vrdef}), and ensures that the integral in \eqref{scalpsir} is well-defined. 

\medskip\noindent
{\bf 3.}
Theorems~\ref{thm:mainalt}--\ref{thm:rfscalingalt} establish the upward LDP and identify basic properties of the rate function $\widehat\psi_r$. The decay of $\widehat\psi_r$ towards zero is \emph{quadratic}, which suggests that $\widehat{\lambda}_N/N$ satisfies a standard central limit theorem. The interpretation of \eqref{curvature} is that the curvature of $\widehat\psi_r$ at its unique zero $C_r$  is the \emph{inverse} of the variance of the associated central limit theorem, in line with standard folklore of large deviation theory (see \cite{B93}). Since $v_r(x)$ can be viewed as the variance of the empirical distribution of the degrees of the vertices with label $\approx xN$ and the large deviations are controlled by $x \in \cD_r$, the relation in \eqref{curvature} is intuitively plausible.   

\medskip\noindent
{\bf 4.}
Assumptions~\eqref{everywhere}--\eqref{infconv} are needed for the upward LDP because the rate is $N$ instead of $N^2$, which requires more delicate arguments. Assumption~\eqref{posdef} is similar to Assumption~\eqref{gap}, since it implies $\lambda_{\max}(\cL_r)\leq\|d_r\|_{\infty}$. Assumption~\eqref{posdef} is needed to make sure that for the upward LDP in Theorem~\ref{thm:mainalt} \emph{the largest degree is dominant}. In fact, we will see that $J(x,\beta)$ in \eqref{Jrdef} is the upward rate function for the degrees of the vertices with label $\approx xN$. Assumption~\eqref{rbd1} is needed to ensure that the general properties of the upward rate function in Theorem~\ref{thm:rfpropalt} hold, Assumption~\eqref{rct} is needed to get the sharp scaling of the upward rate function in Theorem~\ref{thm:rfscalingalt} as stated in \eqref{scal1alt}--\eqref{curvature}.

\medskip\noindent
{\bf 5.}
Note that every graphon of rank 1 satisfies Assumption~\eqref{posdef}. The same is not true for graphons of higher rank. Indeed, consider a block graphon corresponding to a graph that has an adjacency matrix with negative eigenvalues, such as $r(x,y) = \1_{[0,\frac{1}{2}]^2}(x,y) + \1_{[\frac{1}{2},1]^2}(x,y)$. Then $r$ is not non-negative definite.

\medskip\noindent
{\bf 6.}
Whereas the scaling constant in Theorem~\ref{thm:rfscalingalt} is sharp, it is not sharp in Theorem~\ref{thm:rfscaling}. The proof of Theorem~\ref{thm:rfscaling} shows that it lies between $1$ and $2$.

\medskip\noindent
{\bf 7.}
The fine properties of $\psi_r$ and $\widehat\psi_r$ remain elusive. Neither convexity nor analyticity are obvious.

\medskip\noindent
{\bf 8.}
In the proofs we derive an LDP for the maximum degree of $G_N$ and analyse its rate function. The results we obtain for the maximum degree are analogous to Theorems~\ref{thm:main}--\ref{thm:rfscalingalt} and are of independent interest. See Remark~\ref{rem:link}.


\section{Proof of Proposition \ref{prop:basics}}
\label{sec:rf}

Section~\ref{sec:propL} shows that the degree operator is lower semi-continuous. Section~\ref{sec:proofprop} provides the proof of Proposition~\ref{prop:basics}.


\subsection{Lower semi-continuity}
\label{sec:propL}

\begin{lemma}
\label{lem:lsc}
The map $h\mapsto\|\cD_h\|$ is lower semi-continuous in the cut-metric.
\end{lemma}

\begin{proof}
Let $(\ft_n)_{n\in\N}\subset\Wt$ be a sequence converging to some $\ft\in\Wt$ in the cut-metric $\delta_\square$. Without loss of generality we may assume that it converges in the cut-distance $d_\square$ as well, i.e., $\lim_{n\to\infty} d_{\square}(f_n,f) = 0$. Assume that
\begin{equation}
\label{wrongass}
\liminf_{n\to\infty} \|\cD_{f_n}\| < \|\cD_f\|.
\end{equation}
Then there exists a $\delta>0$ (independent of $n$) such that $\|\cD_f\|>\|\cD_{f_n}\|+\delta$ for infinitely many $n\in\N$. By \eqref{eq:degreenorm}, there exists a set $A\subset[0,1]$ of positive measure such that
\begin{equation}
d_f(x) > \|d_f\|_\infty -\tfrac{\delta}{2} = \|\cD_{f}\| - \tfrac{\delta}{2} > \|\cD_{f_n}\| + \tfrac{\delta}{2} \geq d_{f_n}(x) +\tfrac{\delta}{2}
\qquad \forall\,x\in A.
\end{equation}
Hence, by \eqref{degreefunction},
\begin{equation}
\int_{A\times[0,1]} \dx\dy\,f(x,y) > \int_{A\times[0,1]} \dx\dy\,f_n(x,y) + \tfrac{\delta}{2}\lambda(A)
\end{equation}
with $\lambda$ the Lebesgue measure. By the definition of the cut-distance in \eqref{cutdistance},
\begin{equation}
d_{\square}(f_n,f) \geq \int_{A\times[0,1]} \dx\dy\,[f(x,y)-f_n(x,y)] > \tfrac{\delta}{2}\lambda(A).
\end{equation}
Since $\lambda(A)$ and $\delta$ are independent of $n$, this inequality is a contradiction with the fact that $\lim_{n\to\infty} d_\square(f_n,f) = 0$. Hence \eqref{wrongass} is impossible.
\end{proof}


\subsection{Proof of Proposition~\ref{prop:basics}}
\label{sec:proofprop}

\begin{proof}
(i) The first part of the statement follows from the fact that $\cL_h$ is the sum of two bounded operators. For the second part of the statement, we need that $h\mapsto\|\cT_h\|$ is continuous and that $h\mapsto\|\cD_h\|$ is lower semi-continuous. The former was shown in \cite[Lemma 3.6]{LZ16}, the latter was shown in Lemma~\ref{lem:lsc}. 

Let $g\in L^2([0,1]^2)$. Then, by \eqref{Laplaciandef},
\begin{equation}
\begin{split}
&\liminf_{n\to\infty} \big[\|\cL_{f_n}(g)\|_2^2 - \|\cL_{f}(g)\|_2^2\big]\\
&\geq \liminf_{n\to\infty} \big[\|\cD_{f_n}(g)\|_2^2-\|\cD_f(g)\|_2^2\big] 
+ \liminf_{n\to\infty} \big[\|T_{f_n}(g)\|_2^2-\|T_f(g)\|_2^2\big]\\
&\qquad -2 \limsup_{n\to\infty} \big[\langle \cD_{f_n}(g),T_{f_n}(g)\rangle-\langle \cD_{f}(g),T_f(g)\rangle\big]\\
&\geq -2 \limsup_{n\to\infty} \big[\langle \cD_{f_n}(g),T_{f_n}(g)\rangle-\langle \cD_{f}(g),T_f(g)\rangle\big].
\end{split}
\end{equation}
It remains to show that the last expression equals 0. Since the simple functions are dense in $L^2([0,1]^2)$, we may assume without loss of generality that $g$ is simple, i.e., $g=\sum_{i=1}^k\alpha_i\1_{A_i}$. Estimate
\begin{equation}
\begin{split}
&\big|\langle \cD_{f_n}(g),T_{f_n}(g)\rangle-\langle \cD_{f}(g),T_f(g)\rangle\big|\\
&= \left|\int_{[0,1]^3} \dy\dy'\dx\, f_n(x,y)f_n(x,y')g(x)g(y)-\int_{[0,1]^3}\dy\dy'\dx\, f(x,y)f(x,y')g(x)g(y)\right|\\
&= \left|\int_{[0,1]^3}\dy\dy'\dx\, [f_n(x,y)f_n(x,y')-f(x,y)f(x,y')]\,g(x)g(y)\right|\\
&= \left|\sum_{1 \leq i,j \leq k} \alpha_i\alpha_j \int_{[0,1]^3}\dy\dy'\dx\,[f_n(x,y)f_n(x,y')-f(x,y)f(x,y')]\,
\1_{A_i}(x)\1_{A_j}(y)\right|\\
&\leq \sum_{1 \leq i,j \leq k} |\alpha_i\alpha_j| \left(\left|\int_{[0,1]^3}\dy\dx\dy'\, f_n(x,y')[f_n(x,y)-f(x,y)]\,
\1_{A_i}(x)\1_{A_j}(y)\right|\right.\\
&\qquad \qquad \qquad \qquad +\left.\left|\int_{[0,1]^3}\dy'\dx\dy\, f(x,y)[f_n(x,y')-f(x,y')]\,\1_{A_i}(x)\1_{A_j}(y)\right|\right).
\end{split}
\end{equation}
Note that, for every $y'\in[0,1]$, $0\leq f_n(x,y')\1_{A_i}(x),\1_{A_j}(y)\leq1$ and, for every $y\in[0,1]$, $0\leq f(x,y)\1_{A_i}(x)\1_{A_j}(y)\leq1$. Hence, by the definition of the cut-distance, the above expression is bounded from above by
\begin{equation}
\begin{split}
\sum_{1 \leq i,j \leq k} |\alpha_i\alpha_j|\,2d_{\square}(f_n,f),
\end{split}
\end{equation}
which tends to zero as $n\to\infty$.

We have thus shown that $\liminf_{n\to\infty} \|\cL_{f_n}(g)\|_2 \geq \|\cL_{f}(g)\|_2$ for all simple functions $g\in L^2([0,1])$. Now let $\varepsilon>0$. Then there exists a simple $g'\in L^2([0,1])$ such that $\|\cL_f\|<\|\cL_f(g)\|+\varepsilon$. By the previous result,
\begin{equation}
\liminf_{n\to\infty} \|\cL_{f_n}\| \geq \liminf_{n\to\infty} \|\cL_{f_n}(g')\|_2\geq\|\cL_f(g')\| > \|\cL_f\|-\varepsilon.
\end{equation}
Since this holds for arbitrary $\varepsilon>0$, the proof is complete.

\medskip\noindent
(ii) Let $G$ be a graph with $N$ vertices. We prove that $\frac{1}{N}L_N$ and $\cL_{h^{G}}$ have the same eigenvalues. First, let $u=(u_1,\ldots,u_N)$ be an eigenvector of $\frac{1}{N}L_N$ with eigenvalue $\lambda$. Let $u\colon[0,1]\to\R$, and define $\overline{u}\colon[0,1]\to\R$ to be the step function that on $[\frac{i-1}{N},\frac{i}{N})$, $1\leq i\leq N$, equals the average of $u$ over that interval. Then $u$ is an eigenfunction of $\cL_{h^{G}}$ with eigenvalue $\lambda$ if and only if, for all $1\leq j\leq N$ and $x\in[\frac{j-1}{N},\frac{j}{N})$,
\begin{equation}
(\cL_{h^{G}}u)(x) = \frac{1}{N} \sum_{i \in [N] \setminus j} A_N(i,j)u(x)-\frac{1}{N}
\sum_{i \in [N] \setminus j} A_N(i,j)\overline{u}_i=\lambda u(x),
\end{equation}
i.e.,
\begin{equation}
u(x)=\left(\frac{1}{N} \sum_{i \in [N] \setminus j} A_N(i,j) - \lambda\right)^{-1}\frac{1}{N}
\sum_{i \in [N] \setminus j} A_N(i,j)\overline{u}_i.
\end{equation}
Similarly, $v=(v_1,\ldots,v_N)$ is an eigenvector of $\frac{1}{N}L_N$ with eigenvalue $\lambda$ if and only if
\begin{equation}
v_j=\left(\frac{1}{N}\sum_{i \in [N] \setminus j} A_N(i,j) - \lambda\right)^{-1}\frac{1}{N}
\sum_{i \in [N] \setminus j} A_N(i,j)v_i.
\end{equation}
Since every eigenfunction of $\cL_{h^{G}}$ is constant on the blocks $[\frac{j-1}{N},\frac{j}{N})$, $1 \leq j \leq N$, we can view each eigenfunction of $\cL_{h^{G}}$ as an eigenvector of $\frac{1}{N}L_N$, and vice versa.

Finally, note that $\cL_{h^{G}}$ is a normal operator and has a non-negative spectrum. So the norm of $\cL_{h^{G}}$ equals the supremum of its spectrum. Furthermore, the essential spectrum of $\cL_{h^{G}}$ equals the range of the degree function $d_{h^{G}}(x)$ \cite[Proposition 5.11]{DGKR16}. Since the largest eigenvalue of the Laplacian matrix is bounded from below by the maximum degree \cite{GM94}, it follows that the largest eigenvalue of $\cL_{h^{G_N}}$ equals its norm.
\end{proof}


\section{Proof of theorems for downward LDP}

Sections~\ref{sec:proofmain}--\ref{sec:proofrfscaling} provide the proof of Theorems~\ref{thm:main}--\ref{thm:rfscaling}, respectively. 


\subsection{Proof of Theorem \ref{thm:main}}
\label{sec:proofmain}


\subsubsection{Upper bound}

By Proposition~\ref{prop:basics}(i), $h \mapsto \|\cL_h\|$ is lower semi-continuous. Hence the set $\{\widetilde{h} \in \widetilde{\mathcal{W}}\colon\,\|\cL_{\widetilde{h}}\| \leq \beta\}$ is closed in $\widetilde{\mathcal{W}}$. By Proposition~\ref{prop:basics}(ii), $\frac{\|L_N\|}{N} = \|\cL_{h^{G_N}}\|$, and so we can use Theorem~\ref{thm:LDPinhom} in combination with the contraction principle \cite{H00} to get
\begin{equation}
\limsup_{N\to\infty} \binom{N}{2}^{-1} \log \mathbb{P}_N\left(\frac{\|L_N\|}{N} \leq \beta\right) 
\leq - \inf_{ \substack{\widetilde{h} \in \widetilde{\mathcal{W}}\\ \|\cL_{\widetilde{h}}\| \leq \beta} } J_r(\widetilde{h})
= - \inf_{ \substack{h \in \mathcal{W}\\ \|\cL_{h}\| \leq \beta} } I_r(h).
\end{equation}  


\subsubsection{Lower bound}

The proof proceeds via a change of measure. The key is the following lemma.

\begin{lemma}
\label{lem:cont}
Let $h$ be a graphon satisfying \eqref{Assbasic} and $(h_n)_{n\in\N}$ a sequence of block graphons satisfying $\|h_N-h\|_\infty\to0$. Let $h^{G_N}$ be the empirical graphon corresponding to the inhomogeneous Erd\H{o}s-R\'enyi random graph $G_N$ with reference graphon $h_N$. Then $\lim_{N\to\infty} \|\cL_{h^{G_N}}\| = \|\cL_h\|$ in probability.
\end{lemma}   

\begin{proof}
We prove the stronger statement $\lim_{N\to\infty} \|\cL_{h^{G_N}}-\cL_h\| = 0$. Since $\cL_{h}$ is the sum of $\cD_{h}$ and $-\cT_h$ (recall \eqref{Laplaciandef}), it suffices to prove the claim for these two operators separately.

To show that $\cD_{h^{G_N}}$ converges to $\cD_h$ in operator norm in probability, let $d_{h_N}$ be the degree function of the block graphon $h_N$. Then
\begin{equation}
\begin{split}
\|d_{h^{G_N}}-d_h\|_\infty \leq \|d_{h^{G_N}}-d_{h_N}\|_\infty+\|d_{h_N}-d_h\|_\infty.
\end{split}
\end{equation}
Because $\|h_N-h\|_\infty \downarrow 0$ as $N\to\infty$, the second term vanishes. As for the first term,
\begin{equation}
\begin{aligned}
&\P_N(\|d_{h^{G_N}}-d_{h_N}\|_\infty\geq t)\\
&\leq \sum_{i \in [N]} \P_N\left(\left|\frac{1}{N} \sum_{i \in [N] \setminus j}^N A_N(i,j)
-\frac{1}{N}\sum_{i \in [N] \setminus j}^N h_{N,ij}\right|\geq t\right) \leq 2N \eee^{-2Nt^2} \downarrow 0
\end{aligned}
\end{equation}
by a straightforward application of Hoeffding's inequality. We use that $A_N(i,j)$ has Bernoulli distribution with mean $(h_N)_{ij}$. Hence $\|d_{h^{G_N}}-d_{h_N}\|_\infty \to 0$ as $N\to\infty$ in probability, and so
\begin{equation}
\begin{split}
\|\cD_{h^{G_N}}-\cD_h\|=\|d_{h^{G_N}}-d_r\|_\infty \to  0 \text{ in probability}.
\end{split}
\end{equation}

To show that $T_{h^{G_N}}$ converges to $T_h$ in operator norm in probability, it suffices to note that
\begin{equation}
\|T_{h^{G_N}}-T_{h}\| \leq \sqrt{2}\, d_\square(h^{G_N},h) \downarrow 0 \quad \text{ as }N\to\infty \text{ in probability}. 
\end{equation}
The first inequality was shown in \cite[Lemma 3.6]{LZ16}. Convergence in the cut-metric was shown in \cite[Lemma 5.11]{chatterjeeln2017}.
\end{proof}

We are now ready to prove the lower bound.

\begin{proposition}\label{lemma:lowerboundforcontinuous}
Let $\beta\in[0,C_r]$, and let $h\in\W$ be such that $\|\cL_h\|<\beta$. Then
\begin{equation}
\liminf_{N\to\infty} \frac{2}{N^2} \log\P_N(\|\cL_{h^{G_N}}\|<\beta)\geq -I_r(h).
\end{equation}
\end{proposition}

\begin{proof}
The proof comes in three steps.

\medskip\noindent
{\bf 1.}  We begin by giving the proof for the case when $h$ is a \emph{continuous} graphon. Denote the law of an inhomogeneous ERRG with reference graphon $r_N$ by $\P_{N,r_N}$ instead of $\P_{N}$, and the law of an inhomogeneous ERRG with reference graphon $h_N$ by $\P_{N,h_N}$, where $h_N$ is the block graphon that is obtained by averaging $h$ over blocks of size $1/N$. Then, by continuity and hence uniform continuity of $h$, $\|h_N-h\|_\infty \downarrow 0$.

Let $x=\|\cL_h\|$, and define
\begin{equation}
U_\varepsilon^x = \{f\in\W\mid\|\cL_f\|\in(x-\varepsilon,x+\varepsilon)\}.
\end{equation}
By Lemma~\ref{lem:cont},
\begin{equation}
\label{eq:convergenceprobability}
\P_{N,h_N}(U_\varepsilon^x ) \uparrow 1, \qquad N\to\infty,
\end{equation}
for all $\varepsilon>0$. Write
\begin{equation}
\begin{split}
\P_{N,r_N}(U_\varepsilon^x) = \P_{N,h_N}(U_\varepsilon^x)\,\frac{1}{\P_{N,h_N}(U_\varepsilon^x)}
\int_{U_\varepsilon^x}\exp\left(-\log\frac{\d \P_{N,h_N}}{\d\P_{N,r_N}}\right)\,\d\P_{N,h_N}.
\end{split}
\end{equation}
By Jensen's inequality,
\begin{equation}
\log \P_{N,r_N}(U_\varepsilon^x) \geq \log\P_{N,h_N}(U_\varepsilon^x)
-\frac{1}{\P_{N,h_N}(U_\varepsilon^x)}\int_{U_\varepsilon^x} \left(\log\frac{\d\P_{N,h_N}}{\d\P_{N,r_N}}\right) \d\P_{N,h_N}.
\end{equation}
Using \eqref{eq:convergenceprobability} and \cite[Lemma 5.7]{chatterjeeln2017}, we obtain
\begin{equation}
\liminf_{N\to\infty} \frac{2}{N^2} \log\P_{N,r_N}(U_\varepsilon^x)
\geq - \lim_{N\to\infty} \frac{2}{N^2} \int_{U_\varepsilon^x} \left(\log\frac{\d\P_{N,h_N}}{\d\P_{N,r_N}}\right) \d\P_{N,h_N}
=-I_r(h).
\end{equation}
Note that \cite[Lemma 5.7]{chatterjeeln2017} was stated for the homogeneous ERRG, but the proof may be extended to the inhomogeneous ERRG (see \cite[Theorem 6.1]{M20thesis}). Because $U_\varepsilon^x\subseteq\{f\in\W\mid\|\cL_f\|<\beta\}$ for $\varepsilon>0$ small enough, we conclude that
\begin{equation}
\liminf_{N\to\infty}\frac{2}{N^2} \log\P_{N,r_N}(\|\cL_{h^{G_N}}\|<\beta) \geq -I_r(x).
\end{equation}

\medskip\noindent
{\bf 2.} We next extend the proof to the case where $h$ is a \emph{block graphon} such that $\|\cL_h\|<\beta$. Assume that $h$ is constant on the blocks $[\frac{i-1}{N},\frac{i}{N})\times[\frac{j-1}{N},\frac{j}{N})$, $1\leq i,j\leq N$. Then there exists a sequence of continuous graphons $(h_k)_{k\in\N}$ that converges to $h$ in $L^2$ with $h_k\leq h$. Note that, for $f\in\W$,
\begin{equation}\label{eq:normincreasing}
\begin{split}
&\|\cL_f\|\\ 
&=\sup_{\|u\|_2\leq1}\langle g,\cL_f(u)\rangle = \sup_{\|u\|_2\leq1}\int_{[0,1]^2}\d y\,\d x\,u(x)f(x,y) [u(x)-u(y)]\\
&= \tfrac{1}{2}\sup_{\|u\|_2\leq1} \left[\int_{[0,1]^2}\d y\,\d x\,u(x)f(x,y)[u(x)-u(y)]
+\int_{[0,1]^2}\d y\,\d x\,u(y)f(x,y)[u(y)-u(x)]\right]\\
&= \tfrac{1}{2} \sup_{\|u\|^2\leq1}\int_{[0,1]^2}\d x\,\d y\, f(x,y)[u(x)-u(y)]^2,
\end{split}
\end{equation}
where we use that $f$ is a symmetric function. From the above formula it is immediate that, since $h_k\leq h$, we have $\|\cL_{h_k}\| \leq \|\cL_{h}\|<\beta$. Then, by Lemma \ref{lemma:lowerboundforcontinuous},
\begin{equation}
\liminf_{N\to\infty} \frac{2}{N^2} \log\P_N(\|\cL_{h^{G_N}}\|<\beta)\geq -I_r(h_k), \qquad k \geq k_0.
\end{equation}
We conclude by noting that $I_r$ is continuous in $L^2([0,1]^2)$, implying that $I_r(h_k)$ converges to $I_r(h)$ as $k\to\infty$. 

\medskip\noindent
{\bf 3.} We finally extend the proof to the case when $h$ is an \emph{arbitrary} graphon such that $\|\cL_h\|<\beta$. Let $\overline{h}_N$ be the $N$-block approximant of $h$ for some $N$. Then $\overline{h}_N$ converges to $h$ in $L^2$ as $N\to\infty$ \cite[Proposition 2.6]{chatterjeeln2017}. Again, it suffices to prove
\begin{equation}
\|\cL_{\overline{h}_N}\| \leq \|\cL_{h}\|.
\end{equation}
First, suppose that $\|\cL_{\overline{h}_N}\| = \|d_{\overline{h}_N}\|_{\infty}$. For $x\in[\frac{i-1}{N},\frac{i}{N})$,
\begin{equation}
d_{\overline{h}_N}(x) = N \int_{[\frac{i-1}{N},\frac{i}{N})}\d x\, h(x)\leq\|d_h\|_{\infty}.
\end{equation}
It follows that $\|\cL_{\overline{h}_N}\| = \|d_{\overline{h}_N}\|_{\infty} \leq \|d_{h}\|_{\infty} \leq \|\cL_h\|$. Next, suppose that $\|\cL_{\overline{h}_N}\| > \|d_{\overline{h}_N}\|_{\infty}$. Then $\|\cL_{\overline{h}_N}\| = \lambda_{\max}(\cL_{\overline{h}_N})$ and there exists an eigenfunction $u$ of $\cL_h$ with $\|u\|_2=1$ such that $\langle\cL_{\overline{h}_N}u,u\rangle = \|\cL_{\overline{h}_N}\|$. In the proof of Proposition \ref{prop:basics} it was shown that $u$ is constant on each of the intervals $[\frac{i-1}{N},\frac{i}{N})$, $1 \leq i \leq N$. Let $u_i$ be the value of $u$ on $[\frac{i-1}{N},\frac{i}{N})$, $1 \leq i \leq N$. Then
\begin{equation}
\begin{split}
\langle\cL_{h}u,u\rangle =& \int_{[0,1]^2}\dx\,\dy\, h(x,y)[u(x)-u(y)]u(x)\\
=&\sum_{i,j=1}^N \int_{[\frac{i-1}{N},\frac{i}{N})\times[\frac{j-1}{N},\frac{j}{N})}\dx\,\dy\, h(x,y)[u_i-u_j]u_i\\
=&\sum_{i,j=1}^N \int_{[\frac{i-1}{N},\frac{i}{N})\times[\frac{j-1}{N},\frac{j}{N})}\dx\,\dy\, \overline{h}_N(x,y)
[u_i-u_j]u_i=\langle\cL_{\overline{h}_N}u,u\rangle.
\end{split}
\end{equation}
The desired result now follows from the fact that $\|\cL_h\| = \sup_{\|v\|_2=1}\langle\cL_hv,v\rangle \geq \langle\cL_{h}u,u\rangle = \langle\cL_{\overline{h}_N}u,u\rangle = \|\cL_{\overline{h}_N}\|$.
\end{proof}

From Proposition \ref{lemma:lowerboundforcontinuous}, we obtain the lower bound
\begin{equation}
\liminf_{N\to\infty} \frac{2}{N^2} \log\P_N(\|\cL_{h^{G_N}}\|<\beta)\geq -\inf_{\substack{h\in\mathcal{W}\\ \|\cL_h\|<\beta}}I_r(h).
\end{equation}
So it just remains to show
\begin{equation}\label{eq:variationalformulaequality}
\inf_{\substack{h\in\mathcal{W}\\ \|\cL_h\|<\beta}}I_r(h)=\inf_{\substack{h\in\mathcal{W}\\ \|\cL_h\|\leq\beta}}I_r(h).
\end{equation}
Let $h\in\mathcal{W}$ with $\|\cL_h\|=\beta$. Then $\|\cL_{(1-\varepsilon)h}\|=(1-\varepsilon)\beta<\beta$ for all $\varepsilon>0$. Furthermore, $(1-\varepsilon)h$ converges to $h$ in $L^2$ as $\varepsilon\downarrow0$ and $I_r$ is continuous in $L^2$, so $I_r((1-\varepsilon)h)\to I_r(h)$ as $\varepsilon\downarrow0$. Hence, the left-hand side of \eqref{eq:variationalformulaequality} is bounded from above by the right-hand side. The converse inequality is trivial, which settles the proof of the lower bound.


\subsection{Proof of Theorem~\ref{thm:rfprop}}
\label{sec:proofrfprop}

\begin{proof}
(i) Recall from \eqref{eq:normincreasing} in the proof of Proposition \ref{lemma:lowerboundforcontinuous} that if $h_1\leq h_2$, then $\|\cL_{h_1}\| \leq \|\cL_{h_2}\|$, so $h\mapsto\|\cL_h\|$ is increasing. It follows, via the same proof as for \cite[Theorem 1.5(ii)]{CHHS20}, that $\psi_r$ is strictly decreasing on $[0,C_r]$.

Let $t>0$, and let $(t_n)_{n\in\N}$ be a strictly increasing sequence converging to $t$. Since $\Wt$ is compact, and $\h \mapsto I_r(\h)$ and $\h\mapsto\|\cL_{\h}\|$ are lower semi-continuous, for every $n$ there exists a minimiser $\ft_n$ such that $\|\cL_{f_n}\|= t_n$ and $\psi_r(t_n)=J_r(\ft_n)$. Again, by the compactness of $\Wt$, we may assume that the sequence $(\ft_n)_{n\in\N}$ converges to some $\ft\in\Wt$ in the cut-metric. By the lower semi-continuity of $\|\cL_{\h}\|$, we have $\|\cL_{\ft}\|\geq t$. By the lower semi-continuity of $I_r$, we conclude that
\begin{equation}
\liminf_{n\to\infty} \psi_r(t_n) = \liminf_{n\to\infty} J_r(\ft_n) \geq J_r(\ft) \geq \psi_r(t).
\end{equation}
Right-continuity of $\psi_r$ now follows by noting that $\psi_r$ is strictly decreasing. The proof of the left-continuity of $\psi_r$ is the same as the proof of \eqref{eq:variationalformulaequality}.

\medskip\noindent
(ii) This statement is immediate from the lower semi-continuity of $\|\cL_h\|$ and $I_r$, combined with the compactness of $\Wt$.
\end{proof}


\subsection{Proof of Theorem~\ref{thm:rfscaling}}
\label{sec:proofrfscaling} 

\begin{proof}

We first prove a large deviation principle for the degree of a single vertex $\approx xN$ with rate $N$ and with rate function $J_r(x,\beta)$ defined in \eqref{Jrdef}. We subsequently derive upper and lower bounds for $\psi_r(\beta)$ in terms of this rate function, and use these to compute
the scaling of $\psi_r$ near its minimum.


\subsubsection{LDP for single degrees}
\label{sec:LDP}

In this section we prove an LDP for the degree of a single vertex $\approx xN$ subject to \eqref{everywhere}--\eqref{infconv}. The associated rate function is $J_r(x,\beta)$ introduced in \eqref{Jrdef}--\eqref{Cramer}. This rate function will be useful in the proof of Theorem \ref{thm:rfscaling}, even though the latter does not require \eqref{everywhere}--\eqref{infconv}. We prove the degree LDP under these assumptions in order to motivate the connection between the degree LDP and the Laplacian LDP, and because we will need the degree LDP in the proof of Theorem \ref{thm:mainalt}.

Let $(r_N)_{n\in\N}$ be a sequence of block graphons satisfying \eqref{everywhere}--\eqref{infconv}. Let $d_i$ be the degree of vertex $i$. For $x\in[0,1]$, let $i_x \in \{1,\ldots,N\}$ be the index such that $x\in[\frac{i_x-1}{N},\frac{i_x}{N})$. A crucial ingredient is the LDP for the family $N^{-1} d_{i_x}$, $N \in \N$, for fixed $x\in[0,1]$, which we prove with the help of the G\"artner-Ellis theorem. To that end, note that $d_{i_x} = \sum_{j \in [N]} A_{i_xj}$, with $A_{i_xj}=1$ if there is an edge between $i$ and $j$ and $A_{i_xj}=0$ otherwise. Therefore the cumulant generating function of $N^{-1} d_{i_x}$ is
\begin{equation}
\begin{split}
\Lambda_N(x,\theta)
&= \log\EE\left[\exp\left(\theta\frac{1}{N} \sum_{j \in [N]} A_{i_xj}\right)\right]
= \sum_{j \in [N] \setminus i} \log\EE\left[\exp\left(\theta\frac{1}{N} A_{i_xj}\right)\right]\\
&= \sum_{j \in [N]} \log\left((r_{N})_{i_xj}\,\e^{\frac{1}{N}\theta}+[1-(r_N)_{i_xj}]\right)\\
&= N\int_{[0,1]\setminus[\frac{i_x-1}{N},\frac{i_x}{N})}\d y\,\log\left(r_{N}(x,y)\e^{\frac{1}{N}\theta}+[1-r_N(x,y)]\right).
\end{split}
\end{equation}
Since $\|r_N-r\|_\infty\to0$ and $\lim_{N\to\infty}r_N(x,y)$ for all $(x,y)\in[0,1]^2$, we have
\begin{equation}
\label{eq:limitlogMGF}
\Lambda_r(x,\theta) = \lim_{N\to\infty} \frac{1}{N}\Lambda_N(x,N\theta)
=\int_{[0,1]}\d y\,\log\left(r(x,y)\e^\theta+[1-r(x,y)]\right)
\end{equation}
for all $x\in[0,1]$. Since $\theta \mapsto \Lambda_r(x,\theta)$ is finite and differentiable on $\R$, the G\"artner-Ellis theorem tells us that the family $N^{-1}d_{i_x}$, $N\in\N$, satisfies the LDP on $[0,1]$ with rate $N$ and with rate function $x \mapsto J_r(x,\beta)$ given by
\begin{equation}
\label{eq:variationalformularatedegree}
J_r(x,\beta) = \sup_{\theta\in\R} [\theta\beta-\Lambda_r(x,\theta)].
\end{equation}
The supremum is attained at the unique $\theta(x,\beta)$ such that
\begin{equation}
\int_{[0,1]} \d y\,\widehat{r}_\beta(x,y) = \beta \qquad \forall\, x \in [0,1],
\end{equation}
with
\begin{equation}
\widehat{r}_\beta(x,y) = \frac{\e^{\theta(x,\beta)} r(x,y)}{\e^{\theta(x,\beta)} r(x,y)+[1-r(x,y)]}.
\end{equation}
The Lagrange multiplier $\theta(x,\beta)$ exists because $r\in (0,1)$ almost everywhere by \eqref{Assbasic}. A simple computation shows that
\begin{equation}\label{eq:Jrintegral}
J_r(x,\beta)= \int_{[0,1]} \d y\,\mathcal{R}\big(\widehat{r}_\beta(x,y)\mid r(x,y)\big) 
= \sup_u \int_{[0,1]}\d y\, \cR(u(y)\mid r(x,y)),
\end{equation}
where the supremum is taken over all the measurable functions $u\colon[0,1]\to[0,1]$ satisfying $\int_{[0,1]} \d y\, u(y)\leq\beta$. The map
\begin{equation*}
\theta\mapsto \int_{[0,1]} \d y\,\frac{\e^{\theta}r(x,y)}{\e^{\theta}r(x,y)+1-r(x,y)}
\end{equation*}
is a continuous bijection from $\R$ to $(0,1)$, again because $r \in (0,1)$ almost everywhere, and so the inverse map $\beta\mapsto\theta(x,\beta)$ is continuous. 


\subsubsection{Properties of $J_r$}

We need the following lemma for the derivatives of $J_r$ with respect to $\beta$. Henceforth we write $J'_r=\frac{\partial J_r}{\partial\beta}$ and $\theta'=\frac{\partial \theta}{\partial\beta}$, and use the upper index $(k)$ for the $k$-th derivative.

\begin{lemma}
\label{lemma:derivativesJr}
$\beta \mapsto J_r(x,\beta)$ is analytic for every $x \in [0,1]$, with
\begin{equation}
J_r^{(k)}(x,\beta) = \theta^{(k-1)}(x,\beta).
\end{equation}
Moreover, subject to \eqref{rbd1}, 
\begin{equation}
\sup_{\substack{x\in[0,1]\\ \beta\in[\varepsilon,1-\varepsilon]}}|J_r^{(k)}(x,\beta)|<\infty \qquad \forall\,\varepsilon>0.
\end{equation}
\end{lemma}

\begin{proof}
Recall from Section \ref{sec:LDP} that
\begin{equation}
J_r(x,\beta) = \int_{[0,1]} \d y\,\mathcal{R}\big(\widehat{r}_\beta(x,y)\mid r(x,y)\big) 
= \theta(x,\beta)\beta - \int_{[0,1]}\d y\, \log\Big(1+\big(\e^{\theta(x,\beta)}-1\big)r(x,y)\Big).
\end{equation}
Differentiating with respect to $\beta$ and using that $d_{\widehat r_\beta}(x)=\beta$, we obtain
\begin{equation}
J_r'(x,\beta)=\theta(x,\beta).
\end{equation}
Implicit differentiation of the equation
\begin{equation}
\int_{[0,1]}\d y\, \frac{\e^{\theta}r(x,y)}{\e^\theta r(x,y)+(1-r(x,y))} = \beta
\end{equation}
gives
\begin{equation}
\theta'(x,\beta) = \left(\int_{[0,1]}\d y\, \frac{\e^{\theta(x,\beta)} r(x,y)(1-r(x,y))}{\left[\e^{\theta(x,\beta)} r(x,y)+(1-r(x,y))\right]^2}\right)^{-1}.
\end{equation}
Fix $\varepsilon>0$. Since $r$ is bounded away from 0 and 1 by \eqref{rbd1}, we see that $\theta'(x,\beta)$ is bounded for $x\in[0,1]$ and $\beta\in[C_r,1-\varepsilon]$ when $\theta(x,\beta)$ is. Iteratively applying the chain, product and quotient rules of differentiation, we obtain that $\theta^{(k)}$ is some polynomial of order $k$ in the variables
\begin{equation}
\left(\int_{[0,1]} \d y\, \frac{\e^{\theta}r(x,y)(1-r(x,y))}{\left[\e^\theta r(x,y)+(1-r(x,y))\right]^2}\right)^{-1},
\qquad \int_{[0,1]} \d y\,\frac{f\left(\e^{\theta},r,\theta',\ldots,\theta^{(k-1)}\right)}{\left[\e^\theta r(x,y)+(1-r(x,y))\right]^j},
\end{equation} 
with $f$ a polynomial and $j\in\mathbb{N}$. By induction, we obtain that $f$ is bounded when $\theta$ is bounded, and hence that $\theta^{(k)}$ is bounded.

Thus, it remains to show $\sup_{x\in[0,1],\beta\in[\varepsilon,1-\varepsilon]}|\theta(x,\beta)|<\infty$. Let 
\begin {equation}
r_- = \inf_{(x,y)\in[0,1]^2} r(x,y)>0, \qquad \widetilde\theta(\beta) = \log\frac{\beta(1-r_-)}{(1-\beta)r_-}.
\end{equation} 
Since the map $r \mapsto \frac{\e^\theta r}{\e^\theta r+(1-r)}$ is increasing, we have
\begin{equation}
\int_{[0,1]} \d y\, \frac{\e^{\widetilde\theta} r(x,y)}{\e^{\widetilde\theta} r(x,y)+(1-r(x,y))}
\geq \frac{\e^{\widetilde\theta}r_-}{\e^{\widetilde\theta}r_-+(1-r_-)} = \beta.
\end{equation}
Hence, $\theta(x,\beta)\leq\widetilde\theta(\beta)<\infty$. Since $\widetilde\theta(\beta)$ is independent of $x$ and $\sup_{\beta\in[\varepsilon,1-\varepsilon]}|\widetilde\theta(\beta)|<\infty$, all claims are settled.
\end{proof}


\subsubsection{Lower bound}

Let $\bar{h}$ be a minimizer of \eqref{rf}. Then $\int_{[0,1]} \d y\,\bar{h}(x,y) \leq\|\cL_{\bar{h}}\|\leq \beta$ for each $x\in S_r(\beta)$. Hence, by \eqref{eq:Jrintegral},
\begin{equation}
\int_{[0,1]} \d y\,\mathcal{R}\big(\bar{h}(x,y) \mid r(x,y)\big) \geq J_r(x,\beta),
\end{equation}
which implies
\begin{equation}
\label{psilb}
\begin{split}
&\psi_r(\beta) = I_r(\bar{h}) \geq \int_{S_r(\beta)} \d x\,J_r(x,\beta).
\end{split}
\end{equation}


\subsubsection{Upper bound}

Let 
\begin{equation}
h_{\beta}(x,y) = \begin{cases}
r(x,y), & x,y\not\in S_r(\beta)\\
\hat{r}_\beta(x,y), & x \in S_r(\beta)\text{, but } y \not\in S_r(\beta),\\
\hat{r}_\beta(y,x), & y \in S_r(\beta)\text{, but } x \not\in S_r(\beta),\\
\min\{\hat{r}_\beta(x,y),\hat{r}_\beta(y,x)\}, & x,y\in S_r(\beta). 
\end{cases}
\end{equation}
Then $h_{\beta}$ converges to $r$ in $L^\infty$ as $\beta\uparrow C_r$ by Lemma \ref{lemma:derivativesJr}. We show that $\|\cL_{h_\beta}\|\leq\beta$ for $\beta$ close enough to $C_r$. Note that $\|\d_{h_{\beta}}\|_\infty\leq\beta$ by construction, so we only need to show $\lambda_{\max}(\cL_{h_\beta})\leq\beta$ for $\beta$ large enough. Let $\lambda_\beta = \lambda_{\max}(\cL_{h_\beta})$ and let $\lambda'$ be the limit of any convergent subsequence of $(\lambda_\beta)_{k\in\N}$. We show that $\lambda'$ is an eigenvalue of $\cL_h$, so that $\limsup_{\beta\uparrow C_r}\lambda_\beta \leq \lambda_{\max}(\cL_r)\leq\beta$ for $\beta$ close enough to $C_r$. Here we use Assumption \eqref{gap}. 

Assume, by contradiction, that $\lambda'$ is not an eigenvalue of $\cL_r$. Without loss of generality, we may assume that $\lambda_\beta$ converges to $\lambda'$. Consider $F\colon\, L^\infty([0,1]^2) \times \R \times L^2([0,1]) \to L^2([0,1])$ given by $F(g,\mu,u) = \cL_{g}u-\mu u$. This map is bounded and affine in each coordinate, and hence is continuously Fr\'echet differentiable. The Fr\'echet derivative of $F$ at a point $(g,\mu,u)$ is given by 
\begin{equation}
((DF)(g,\mu,u))(f,\nu,w)=F(f,\nu,u)+F(g,\mu,w).
\end{equation}
Indeed, let $(f_k,\nu_k,w_k)$ such that $\|(f_k,\nu_k,w_k)\|=\|f_k\|_{\infty}+|\nu_k|+\|w_k\|_2\to0$ as $k\to\infty$. Then
\begin{equation}
\begin{split}
&\frac{\left\|F(g+f_k,\mu+\nu_k,u+w_k)-F(g,\mu,u)-((DF)(g,\mu,u))(f_k,\nu_k,w_k)\right\|}{\|(f_k,\nu_k,w_k)\|}\\
=&\frac{\|F(f_k,\nu_k,w_k)\|}{\|(f_k,\nu_k,w_k)\|}\leq\frac{\|\cL_{f_k}w_k\|_2+\|\nu_k w_k\|_2}{\|f_k\|_\infty+|\nu_k|+\|w_k\|_2}\leq\frac{2\|f_k\|_{\infty}\|w_k\|_2+|\nu_k|\|w_k\|_2}{\|f_k\|_\infty+|\nu_k|+\|w_k\|_2}\to0,\qquad k\to\infty.
\end{split}
\end{equation}
Note that $F$ is not necessarily Fr\'echet differentiable as a function $L^2([0,1]^2) \times \mathbb{R} \times L^2([0,1]) \to L^2([0,1])$. Since $\lambda'$ is not an eigenvalue of $\cL_r$, the map
\begin{equation}
w\mapsto ((DF)(r,\lambda',0))(0,0,w)=F(r,\lambda',w)
\end{equation}
has a trivial kernel and so is an isomorphism on its image space. So, by the implicit function theorem for Banach spaces \cite[Theorem I.5.9]{L99}, there exists a neighbourhood $U$ of $(r,\lambda')  \in L^\infty([0,1]^2) \times \R$ and a neighbourhood $V$ of $\underline{0} \in L^2([0,1])$ such that $F(g,\mu,u)=0$ if and only if $u=0$ for all $(g,\mu,u) \in U \times V$. Since $h_\beta \to r$ in $L^\infty$ and $\lambda_\beta \to \lambda'$ as $\beta\uparrow C_r$, we have that $(h_\beta,\lambda_\beta) \in U$ for $\beta \geq \beta_0$. However, since $\lambda_\beta$ is an eigenvalue of $\cL_{h_\beta}$, there exists an eigenfunction $u_\beta \in L^2([0,1])$ with $u_\beta \neq 0$ such that $F(h_\beta,\lambda_\beta,u_\beta) = 0$. Since $V$ is a neighbourhood of 0, we can rescale $u_\beta$ such that it lies in $V$. This yields a contradiction, and so $\lambda'$ must be an eigenvalue of $\cL_r$.

Now since $\|\cL_{h_\beta}\|\leq\beta$ for $\beta$ sufficiently close to $C_r$, we have
\begin{equation}
\label{niceint}
\begin{split}
\psi_r(\beta)&\leq I_r(h_\beta)\\
&= 2 \int_{S_r(\beta)\times([0,1]\setminus S_r(\beta))} \dx\,\dy\,
\cR\big(\hat{r}_{\beta}(x,y)\mid r(x,y)\big)\\
&\qquad + \int_{S_r(\beta)^2}\dx\,\dy\, \cR\big(\min\{\hat{r}_{\beta}(x,y),\hat{r}_{\beta}(y,x)\}\mid r(x,y)\big)\\
& \leq 2\int_{S_r(\beta)\times([0,1]\setminus S_{\beta})} \dx\,\dy\,
\cR\big(\hat{r}_{\beta}(x,y)\mid r(x,y)\big) + \int_{S_r(\beta)^2}\dx\,\dy\, \cR\big(\hat{r}_{\beta}(x,y)\mid r(x,y)\big)\\
&\qquad + \int_{S_r(\beta)^2}\dx\,\dy\, \cR\big(\hat{r}_{\beta}(y,x)\mid r(x,y)\big)\\
&= 2\int_{S_r(\beta)}\dx\, J_r(x,\beta), \qquad  \beta\uparrow C_r.
\end{split}
\end{equation}
For the last equality, we use that $r$ is symmetric and that $S_r(\beta)^2$ is a symmetric domain.

\begin{remark}
\label{rem:link}
{\rm The above shows that, up to a constant, the last integral in \eqref{niceint} also equals the rate function in the LDP for the maximum degree. The upper bound can also be shown directly by the following computation. 

The state space for the edges determining the adjacency matrix $(A_{ij})$ is $\{0,1\}^{\binom{N}{2}}$, which we endow with the standard partial ordering. The probability distribution of the edges is the product measure $\prod_{\{i,j\}} \mathrm{BER}(r_N(\tfrac{i}{N},\tfrac{j}{N}))$, which is log-convex. Note that the events $\{N^{-1} \max_{1 \leq i \leq N-1} d_i $ $\le \beta \}$ and $\{N^{-1} d_N \le \beta\}$ are non-increasing in the partial ordering. Hence we can use the FKG-inequality iteratively, to get
\begin{equation}
\P\left(N^{-1} \max_{i \in [N]} d_i \leq \beta\right) \ge \prod_{i \in [N]} \P(N^{-1} d_i \leq \beta). 
\end{equation}
Consequently,
\begin{equation}
\label{Jxbeta}
\begin{aligned}
&\frac{1}{\binom{N}{2}} \log \P\left(N^{-1} \max_{i \in [N]} d_i \leq \beta\right) 
\geq \frac{1}{\binom{N}{2}} \sum_{i=1}^N \log \P(N^{-1} d_i \leq \beta)\\
&= \frac{2N}{N-1} \frac{1}{N} \sum_{i \in [N]} \frac{1}{N} \log \P(N^{-1} d_i \leq \beta)
= \frac{2N}{N-1} \int_{[0,1]} \d x\,\frac{1}{N} \log \P(N^{-1} d_{i_x} \leq \beta).
\end{aligned}
\end{equation}
By the law of large numbers, $\lim_{N\to\infty} N^{-1} d_{i_x} = \int_{[0,1]} \d y\,r(x,y) = d_r(x)$ $\P$-a.s., and so in the limit as $N\to\infty$ the last integral in \eqref{Jxbeta} may be restricted to the set $S_r(\beta) = \{x \in [0,1]\colon\,d_r(x) \geq \beta\}$, i.e.,
\begin{equation*}
\liminf_{N\to\infty} \frac{1}{\binom{N}{2}} \log \P\left(N^{-1} \max_{i \in [N]} d_i \leq \beta\right) 
\geq - 2 \int_{S_r(\beta)} \d x\,J_r(x,\beta),
\end{equation*}
where we use that the family $N^{-1} d_{i_x}$, $N\in\N$, satisfies the LDP on $[0,1]$ with rate $N$ and with rate function $x \mapsto J_r(x,\beta)$, as shown in Section~\ref{sec:LDP}.
}\hfill$\spadesuit$
\end{remark}


\subsubsection{Scaling of $J_r(x,\beta)$}
\label{sec:scalingJr}

Via Lemma \ref{lemma:derivativesJr} it is straightforward to show that $J_r(x,d_r(x))=J'_r(x,d_r(x))=0$ and $J''_r(x,d_r(x))=\frac{1}{v_r(x)}$. So, by Taylor expansion,
\begin{equation}
J_r(x,\beta) = \frac{1}{2v_r(x)}(\beta-d_r(x))^2 + \frac{1}{6}\theta''(x,\beta_*)(\beta-d_r(x))^3
\end{equation}
for some $\beta_*=\beta_*(x)\in[\beta,d_r(x)]$. Inserting this expansion into the lower and upper bounds derived above and using that $\theta''(x,\beta)$ is bounded uniformly in $x$ and $\beta$, we obtain
\begin{equation}
\psi_r(\beta)\asymp \int_{S_r(\beta)}\d x\, J_r(x,\beta)\sim\int_{S_r(\beta)}\d x\, \frac{1}{2v_r(x)}(\beta-d_r(x))^2, \qquad \beta\uparrow C_r.
\end{equation}
\end{proof}


\section{Proof of theorems for upward LDP}

Sections~\ref{sec:proofmainalt}--\ref{sec:proofrfscalingalt} provide the proof of Theorems~\ref{thm:mainalt}--\ref{thm:rfscalingalt}, respectively. 


\subsection{Proof of Theorem~\ref{thm:mainalt}}
\label{sec:proofmainalt}

\begin{proof}
We first prove the LDP for the maximum degree. Afterwards we prove the LDP for the maximal eigenvalue.

\subsubsection{LDP for the maximum degree}

For $x\in[0,1]$, let $i_x \in \{1,\ldots,N\}$ be the index such that $x\in[\frac{i_x-1}{N},\frac{i_x}{N})$. Let $d_i$ be the degree of vertex $i$. First note that we can sandwich
\begin{equation}
\label{eq:estimate}
\begin{split}
\sup_{x\in[0,1]} \P(N^{-1}d_{i_x} \geq \beta)
&\leq \P\left(N^{-1}\max_{i \in [N]} d_i \geq \beta\right)\\
&\leq \sum_{i \in [N]} \P(N^{-1} d_i \geq \beta)
\leq N \sup_{x\in[0,1]} \P(N^{-1} d_{i_x} \geq \beta).
\end{split}
\end{equation}
Using the LDP for the single degrees derived in Section~\ref{sec:LDP} subject to \eqref{everywhere}--\eqref{infconv}, we obtain
\begin{equation}
\label{eq:rfmax}
\begin{split}
&\lim_{N\to\infty} \frac{1}{N} \log\P\left(N^{-1} \max_{i \in [N]}d_i\geq\beta\right)
= \lim_{N\to\infty} \sup_{x\in[0,1]} \frac{1}{N} \log\P\left(N^{-1} d_{i_x}\geq\beta\right)\\
&= \sup_{x\in[0,1]} \lim_{N\to\infty} \frac{1}{N}\log\P\left(N^{-1} d_{i_x}\geq\beta\right)
= - \inf_{x\in[0,1]} J_r(x,\beta) = -\widehat{\psi}_r(\beta).
\end{split}
\end{equation}
The convergence of $\Lambda(\theta)$ in \eqref{eq:limitlogMGF} is uniform in $x\in[0,1]$ because $\|r_N-r\|_\infty \to 0$ as $N\to\infty$. Hence the convergence in the LDP for $N^{-1} d_{i_x}$ is uniform in $x\in[0,1]$. This allows us to swap the supremum and the limit in the second equality.


\subsubsection{LDP for $\lambda_{\max}(L_N)$}

By Weyl's interlacing inequalities, we have $\lambda_{\max}(L_N)\leq\lambda_{\max}(D_N)-\lambda_{\min}(A_N)$. We also know that $\lambda_{\max}(D_N)\leq\lambda_{\max}(L_N)$. Hence
\begin{equation}
\label{eq:interlacing}
\lambda_{\max}(D_N) \leq \lambda_{\max}(L_n) \leq \lambda_{\max}(D_N) - \lambda_{\min}(A_n).
\end{equation}
Since $r$ is non-negative definite by \eqref{posdef}, the smallest eigenvalue of $T_r$ is 0. Now let $(T_n)_{n\in\N}$ be a sequence of self-adjoint operators acting on $L^2([0,1])$ and converging in operator norm to some operator $T$. Then
\begin{equation}
\left|\lambda_{\min}(T_n)-\lambda_{\min}(T)\right|
= \left|\inf_{\|f\|_2=1}\langle f,T_nf\rangle-\inf_{\|f\|_2=1}\langle f,Tf\rangle\right|
\leq \sup_{\|f\|_2=1}|\langle f,T_nf\rangle-\langle f,Tf\rangle| \to 0
\end{equation}
as $n\to\infty$. Thus, the map $h\mapsto\lambda_{\min}(\cT_h)$ is continuous in the cut norm. Hence, by the contraction principle, $N^{-1}\lambda_{\min}(A_N) = \lambda_{\min}(T_{h^{G_N}})$ converges to $\lambda_{\min}(T_r) = 0$ at an exponential rate with coefficient $N^2$. By \eqref{eq:rfmax} and \eqref{eq:interlacing} we can sandwich, for all $\varepsilon>0$,
\begin{equation}
\begin{aligned}
- \widehat{\psi}_r(\beta) 
&= - \inf_{x\in[0,1]} J_r(x,\beta) 
= \lim_{N\to\infty}\frac{1}{N}\log\P\left(N^{-1} \max_{i \in [N]}d_i\geq\beta\right)\\
&\leq \lim_{N\to\infty} \frac{1}{N}\log\P\big(N^{-1}\lambda_{\max}(L_N)\geq\beta\big)\\
&\leq \lim_{N\to\infty} \frac{1}{N}\log\left[\P\left(N^{-1}\max_{i \in [N]} d_i\geq\beta-\varepsilon\right)
+\P\big(N^{-1}\lambda_{\min}(A_N)\leq-\varepsilon\big)\right]\\
&= \lim_{N\to\infty} \frac{1}{N }\log\P\left(N^{-1}\max_{i \in [N]} d_i\geq\beta-\varepsilon\right)
= - \inf_{x\in[0,1]} J_r(x,\beta-\varepsilon) 
= - \widehat{\psi}_r(\beta-\varepsilon).
\end{aligned}
\end{equation}
The result now follows from continuity of $\widehat\psi_r$, as stated in Theorem \ref{thm:rfpropalt}. Note that the proof of Theorem \ref{thm:rfpropalt} does not use Theorem \ref{thm:mainalt}.
\end{proof}


\subsection{Proof of Theorem~\ref{thm:rfpropalt}}
\label{sec:proofrfpropalt}


\subsubsection{Proof that $\widehat\psi_r$ is strictly increasing}

Assume that $\inf_{x\in[0,1]} \theta(x,\beta)>0$ for all $\beta>C_r$. Then, by Lemma \ref{lemma:derivativesJr}, for $\beta_2>\beta_1>C_r$,
\begin{equation}
\begin{split}
\widehat\psi_r(\beta_2) =& \inf_{x\in[0,1]}J_r(x,\beta_2)\\
\geq\,& \inf_{x\in[0,1]}J_r(x,\beta_1)+\theta(x,\beta_1)(\beta_2-\beta_1)\\
\geq\,& \widehat\psi_r(\beta_1) + \inf_{x\in[0,1]}\theta(x,\beta_1)(\beta_2-\beta_1)>\widehat\psi_r(\beta_1).
\end{split}
\end{equation}
For the first inequality we use that $J'_r(x,\beta)=\theta(x,\beta)$ is increasing in $\beta$. Hence, it suffices to show $\inf_{x\in[0,1]}\theta(x,\beta)>0$ for all $\beta>C_r$. 

Note that the map $r\mapsto \frac{\e^\theta r}{1+(\e^\theta-1)r}$ is concave. Let $\widetilde\theta(x,\beta)=\log\frac{\beta(1-d_r(x))}{(1-\beta)d_r(x)}$. Then, by Jensen's inequality, for all $x\in[0,1]$,
\begin{equation}
\int_{[0,1]} \d y\, \frac{e^{\widetilde\theta(x,\beta)} r(x,y)}{1+(\e^{\widetilde\theta(x,\beta)}-1)r(x,y)}\leq\frac{\e^{\widetilde\theta(x,\beta)} d_r(x)}{1+(\e^{\widetilde\theta(x,\beta)}-1)d_r(x)}=\beta.
\end{equation}
Recall that $\theta(\beta,x)$ is chosen such that 
\begin{equation}
\int_{[0,1]} \d y\, \frac{\e^{\theta(x,\beta)} r(x,y)}{1+(\e^{\theta(x,\beta)}-1)r(x,y)}=\beta.
\end{equation} 
This implies that $\theta(x,\beta)\geq\widetilde\theta(x,\beta)\geq\log\frac{\beta(1-C_r)}{(1-\beta)C_r}$. We conclude the proof that $\widehat\psi_r$ is strictly increasing by noting that this lower bound is strictly positive for $\beta>C_r$ and is independent of $x$.


\subsubsection{Continuity of $\widehat\psi_r$}

Right-continuity of $\widehat\psi_r$ follows from the fact that the infimum of continuous functions is right-continuous. By Lemma \ref{lemma:derivativesJr}, $\sup_{x\in[0,1]}\theta(x,\beta)<\infty$ for all $\beta\geq C_r$. Thus,
\begin{equation}
\begin{split}
\widehat\psi_r(\beta-\varepsilon) =& \inf_{x\in[0,1]}J_r(x,\beta-\varepsilon)\\
\geq\,& \inf_{x\in[0,1]}[J_r(x,\beta)-\theta(x,\beta)\varepsilon]\\
\geq\,& \widehat\psi_r(\beta)-\sup_{x\in[0,1]}\theta(x,\beta)\varepsilon\to\widehat\psi_r(\beta), \qquad \varepsilon\downarrow 0.
\end{split}
\end{equation}


\subsubsection{Value of $\widehat\psi_r$ at the boundary}

We first show that $\widehat\psi_r(C_r)=0$. By Lemma \ref{lemma:derivativesJr},
\begin{equation}
\begin{split}
\widehat\psi_r(C_r) =& \inf_{x\in[0,1]}\left[J_r(x,d_r(x))+(C_r-d_r(x))\sup_{\beta\in[d_r(x),C_r]}\theta(x,\beta)\right]\\
\leq&\, \inf_{x\in[0,1]}(C_r-d_r(x))\sup_{\beta\in[d_r(x),C_r]}\theta(x,\beta)=0,
\end{split}
\end{equation}
where we use that $C_r=\sup_{x\in[0,1]}d_r(x)$ and that $\theta(x,\beta)$ is bounded uniformly in $x$ and $\beta$. Since
$\hat r_1(x,y)\equiv 1$, we have
\begin{equation}
\widehat\psi_r(1)=\inf_{x\in[0,1]}J_r(x,1)=\inf_{x\in[0,1]}\int_{[0,1]}\d y\, \log\frac{1}{r(x,y)}.
\end{equation}


\subsection{Proof of Theorem~\ref{thm:rfscalingalt}}
\label{sec:proofrfscalingalt}

Recall from Section \ref{sec:scalingJr} that
\begin{equation}
J_r(x,\beta) = \frac{1}{2v_r(x)}\,(\beta-d_r(x))^2 + \frac{1}{6}\,\theta''(x,\beta_*)\,(\beta-d_r(x))^3,
\end{equation}
for some $\beta_*=\beta_*(x)\in[d_r(x),\beta]$. By Lemma \ref{lemma:derivativesJr}, $\theta''$ is bounded uniformly in $x$ and also uniformly in $\beta$ bounded away from 1, so it immediately follows that 
\begin{equation}
\hat\psi_r(\beta) \leq \inf_{x\in\cD_r}J_r(x,\beta) \leq \inf_{x\in\cD_r}\frac{1}{2v_r(x)}(\beta-C_r)^2[1+o(1)], \qquad \beta\downarrow C_r.
\end{equation}
Let $x_*=x_*(\beta)=\arg\min_{x\in[0,1]}J_r(x,\beta)$. It is clear from the above that $\beta-C_r\leq\beta-d_r(x_*)=O(\beta-C_r)$. Hence
\begin{equation}
\hat\psi_r(\beta) = J_r(x_*,\beta) \geq \frac{1}{2v_r(x_*)}(\beta-C_r)^2+O((\beta-C_r))^3), 
\qquad \beta\downarrow C_r.
\end{equation}
By continuity of $r$, $\inf_{x\in\cD_r}|x-x_*| \downarrow 0$ as $\beta\downarrow C_r$. Since also $v_r$ is continuous, this implies that $\frac{1}{2v_r(x_*)}\geq \inf_{x\in\cD_r}\frac{1}{2v_r(x)}[1+o(1)]$ as $\beta\downarrow C_r$. We conclude that 
\begin{equation}
\hat\psi_r(\beta) \geq \inf_{x\in\cD_r}\frac{1}{2v_r(x)}(\beta-C_r)^2[1+o(1)], \qquad \beta\downarrow C_r,
\end{equation}
which settles the claim.




\end{document}